\documentclass[english,reqno]{amsart}
\usepackage{etex}
\usepackage{amsmath,amssymb,amsthm,bbm,mathtools,comment}
\usepackage[shortlabels]{enumitem}
\usepackage[pdftex,colorlinks,backref=page,citecolor=blue]{hyperref}
\hypersetup{pdfpagemode=UseNone,pdfstartview={XYZ null null 1.00}}
\usepackage[mathscr]{euscript}
\usepackage[usenames,dvipsnames]{color}
\usepackage{adjustbox,tikz,calc,graphics,babel,standalone}
\usetikzlibrary{decorations.pathreplacing,calc,arrows,arrows.meta,patterns}
\usepackage[final]{microtype}
\usepackage[numbers]{natbib}
\usepackage{cmtiup}
\usepackage{amsfonts}
\usepackage{graphicx}
\usepackage{caption}
\usepackage{subcaption}
\usepackage{verbatim}
\usepackage{array}
\usepackage[frame,cmtip,arrow,matrix,line,graph,curve]{xy}
\usepackage{graphpap, color, pstricks}
\usepackage{pifont}
\usepackage[final]{microtype}
\usepackage{cmtiup}

\allowdisplaybreaks

\theoremstyle{plain}
\newtheorem{theorem}{Theorem}[section]		
\newtheorem{lemma}[theorem]{Lemma}

\newtheorem{proposition}[theorem]{Proposition}

\newtheorem{conjecture}[theorem]{Conjecture}

\theoremstyle{remark}

\def\CC{\mathcal{C}}
\newcommand{\eps}{\ensuremath{\varepsilon}}

\def\E{\mathbb{E}}
\def\N{\mathbb{N}}

\def\CS{\mathcal{S}}
\def\CL{\mathcal{L}}
\def\CD{\mathcal{D}}

\def\SS{\mathscr{S}}

\let\originalleft\left
\let\originalright\right
\renewcommand{\left}{\mathopen{}\mathclose\bgroup\originalleft}
\renewcommand{\right}{\aftergroup\egroup\originalright}

\makeatletter
\def\imod#1{\allowbreak\mkern10mu({\operator@font mod}\,\,#1)}
\makeatother

\begin{document}
\title{A universal exponent for homeomorphs}
\author{Peter Keevash}
\address{Mathematical Institute, University of Oxford, Andrew Wiles Building, Radcliffe Observatory Quarter, Woodstock Road, Oxford OX2\thinspace6GG, UK}
\email{keevash@maths.ox.ac.uk}

\author{Jason Long}
\address{Mathematical Institute, University of Oxford, Andrew Wiles Building, Radcliffe Observatory Quarter, Woodstock Road, Oxford OX2\thinspace6GG, UK}
\email{jlong@maths.ox.ac.uk}

\author{Bhargav Narayanan}
\address{Department of Mathematics, Rutgers University, Piscataway, NJ 08854, USA}
\email{narayanan@math.rutgers.edu}

\author{Alex Scott}
\address{Mathematical Institute, University of Oxford, Andrew Wiles Building, Radcliffe Observatory Quarter, Woodstock Road, Oxford OX2\thinspace6GG, UK}
\email{scott@maths.ox.ac.uk}

\date{24 March, 2020}

\maketitle
\begin{abstract}
We prove a uniform bound on the topological Tur\'an number of an arbitrary two-dimensional simplicial complex $\SS$: any $n$-vertex two-dimensional complex with at least $C_\SS n^{3-1/5}$ facets contains a homeomorphic copy of $\SS$, where $C_\SS > 0$ is an absolute constant depending on $\SS$ alone. This result, a two-dimensional analogue of a classical result of Mader for one-dimensional complexes, sheds some light on an old problem of Linial from 2006.
\end{abstract}

\section{Introduction}
A number of natural extremal geometric problems arise when we view an $r$-uniform hypergraph as an $(r-1)$-dimensional simplicial complex (by identifying edges with facets). Questions of this nature arise in the high-dimensional combinatorics programme of Linial~\citep{linial1, linial2}, and have also been raised by Gowers~\citep{tim2}; for a sample of some  recent results in this programme, see~\citep{spheres, threshold, linial3, linial4}. In this paper, we study the Tur\'an problem for 2-complexes, or equivalently, the topological Tur\'an problem for 3-graphs.

In the Tur\'an theory of 3-graphs (see~\citep{turan1, turan2}), one is concerned with finding a copy of a fixed 3-graph as a \emph{subgraph}; in the context of 2-complexes, the appropriate replacement for the notion of a subgraph is that of a \emph{homeomorphic image}. More formally, we say that two 3-graphs $G$ and $H$ are \emph{homeomorphic} if they are homeomorphic as topological spaces (when viewed as 2-complexes), and we say that \emph{$G$ contains a homeomorph of $H$} if there is a subgraph of $G$ homeomorphic to $H$. The following example may help clarify this point of view: a 3-graph $H$ is a homeomorph of the complete 3-graph $K^3_4$ on four vertices (or equivalently, the two-dimensional sphere $S^2$) if we can place the vertices of $H$ on the sphere and then triangulate the sphere using those vertices in such a way that the resulting triangles are precisely the edges of $H$.  In this language, our main contribution is the following theorem.

\begin{theorem}\label{thm:main}
For each $3$-graph $H$, there exists $C_H>0$ such that any $3$-graph $G$ on $n$ vertices with at least $C_H n^{3-1/5}$ edges contains a homeomorph of $H$.
\end{theorem}

Here, it is worth mentioning that the topological Tur\'an problem for $2$-graphs (i.e., graphs) is understood reasonably well: a classical result of Mader~\citep{mader1} asserts that for any graph $H$, there exists $C_H >0$ such that every $n$-vertex graph with at least $C_H n$ edges contains a homeomorph of $H$, and this is tight in general up to the multiplicative constant. Linial~\citep{linialques1, linialques2} has raised the question of an analogous result for $3$-graphs, and while answers are available for a few \emph{specific} $3$-graphs, no general results for $3$-graphs in the spirit of Mader's theorem appear to have been previously known; our main result, Theorem~\ref{thm:main}, fills in this gap.

We shall in fact prove Theorem~\ref{thm:main} with $C_H=2000v(H)^6$ for all sufficiently large $n \in \N$. However, we make no attempt to optimise this constant since we do not believe the exponent of $3-1/5$ in our result to be tight; instead, we expect the right exponent to be $3-1/2 = 5/2$, and make the following conjecture. 
\begin{conjecture}\label{mainconj}
For each $3$-graph $H$, there exists $C_H >0$ such that any $3$-graph $G$ on $n$ vertices with at least $C_H n^{5/2}$ edges contains a homeomorph of $H$.
\end{conjecture}
This conjectural exponent of $5/2$ requires explanation, and this brings us to the starting point of the line of investigation we pursue in this paper. In the specific case of the tetrahedron $K_4^3$, a classical result of Brown, Erd\H{o}s and S\'os~\citep{bes} says that $5/2$ is indeed the correct exponent: the minimum number of edges guaranteeing a homeomorph of the sphere in an $n$-vertex $3$-graph is $\Theta(n^{5/2})$. Conjecture~\ref{mainconj} is then motivated by the following line of reasoning: it turns out that we may find homeomorphs in 2-graphs roughly once we are able to find \emph{cycles}, i.e., homeomorphs of $S^1$, and our investigations suggest that we ought to be able to find homeomorphs in 3-graphs roughly once we are able to find \emph{spheres}, i.e., homeomorphs of $S^2$. 

The arguments of Brown, Erd\H{o}s and S\'os are however rather specific to the sphere and, slightly more generally, to `double-pyramidal' complexes. Consequently, even the specialisation of Conjecture~\ref{mainconj} to specific $3$-graphs leads to interesting questions; indeed, the special case of $H$ being (a triangulation of) the torus remains open, and has been reiterated by Linial~\citep{linialques1, linialques2} on multiple occasions as a natural starting point.

To put Theorem~\ref{thm:main} in context, another (cheap) argument is worth mentioning: for any 3-graph $H$, it is not difficult to construct a \emph{3-partite} 3-graph $\tilde H$ that is homeomorphic to $H$ (as shown in Figure~\ref{fig:subd}), and since finding a copy of $\tilde H$ as a \emph{subgraph} is a degenerate Tur\'an problem, it follows from a classical result of Erd\H{o}s~\citep{degen} that there is an $\eps_H > 0$ such that any $n$-vertex 3-graph with at least $n^{3-\eps_H}$ edges contains a copy of $\tilde H$ as a subgraph, and hence a homeomorphic copy of $H$. In contrast, Theorem~\ref{thm:main} says that this $H$-specific exponent $\eps_H$ may actually be replaced by a universal exponent of $1/5$. 

The level of generality at which Theorem~\ref{thm:main} applies comes at a price, however: for a few specific 3-graphs of interest, such as the sphere and the torus for example, the aforementioned arguments (i.e., that of Brown--Erd\H{o}s--S\'os, and the one based on the degenerate 3-graph Tur\'an problem) yield better estimates than what is promised by Theorem~\ref{thm:main}.

\begin{figure}
	\centering
	\begin{tikzpicture}[scale=5, every node/.style={scale=1}]
	\draw[fill,gray,opacity=0.1] (0,0)--(1,0)--(1/2,0.866)--(0,0);
	\draw (0,0)--(1,0)--(1/2,0.866)--(0,0);
	\draw (0.5,-0.15) node[] {$ H$};
	\draw (2.5,-0.15) node[] {$ \tilde H$};

	\draw[fill,gray,opacity=0.1] (2,0)--(3,0)--(2.5,0.866)--(2,0);
	\draw (2,0)--(3,0)--(2.5,0.866)--(2,0);
	\draw (0,0) node[] {$\bullet$};
	\draw (1,0) node[] {$\bullet$};
	\draw (1/2,0.866) node[] {$\bullet$};
	
	\draw (2.25,0.866/2)--(2.75,0.866/2)--(2.5,0)--(2.25,0.866/2);
	\draw (2.25,0.866/2)--(2.5,0.289)--(2.75,0.866/2);
	\draw (2.5,0.289)--(2.5,0);

	\draw (2,0)--(2.5,0.289);
	\draw (3,0)--(2.5,0.289);
	\draw (2.5,0.866)--(2.5,0.289);
	
	\draw (2.25,0.866/2) node[blue] {$\bullet$};
	\draw (2.75,0.866/2) node[blue] {$\bullet$};
	\draw (2.5,0) node[blue] {$\bullet$};
	
	\draw (2.375,0.2165) node[green] {$\bullet$};
	\draw (2.625,0.2165) node[green] {$\bullet$};
	\draw (2.5,0.866/2) node[green] {$\bullet$};
	
	\draw (2,0) node[red] {$\bullet$};
	\draw (3,0) node[red] {$\bullet$};
	\draw (2.5,0.866) node[red] {$\bullet$};
	\draw (2.5,0.289) node[red] {$\bullet$};

	\draw[->,  line width=0.3mm] (1.2,0.433) -- (1.8,0.433);
	\end{tikzpicture}
	\caption{Each edge of $H$ maps to twelve new edges in $\tilde H$; the colours red, blue and green describe the tripartition of $\tilde H$.}\label{fig:subd}
\end{figure}
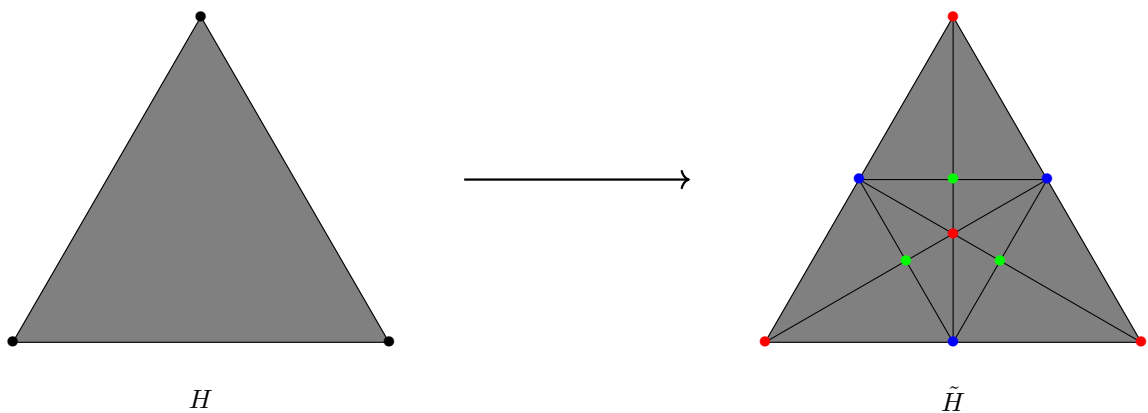

This paper is organised as follows. We begin with some definitions and establish some of the basic notions we need for the proof of our main result in Section~\ref{sec:prelim}. The proof of Theorem~\ref{thm:main} then follows in Section~\ref{sec:proof}. We conclude with a discussion about the limits of our approach, as well as some open problems, in Section~\ref{sec:conc}.

\section{Preliminaries}\label{sec:prelim}
Our notation is for the most part standard. Given a $2$-graph or a $3$-graph $G$, we write $v(G)$ and $e(G)$ for the number of vertices and edges of $G$ respectively. For a set $S$ of vertices in a 2-graph $G$, we write $\Gamma(S)$ for the set of common neighbours of $S$ in $G$, and following a common abuse, we write $\Gamma(x)$ for $\Gamma(\{x\})$, $\Gamma(x,y)$ for $\Gamma(\{x,y\})$, and so on; in the sequel, whenever we refer to $\Gamma(\cdot)$, the underlying graph will always be clear from the context, so there should be no cause for confusion. In those arguments that will involve working with both 2-graphs and 3-graphs in close proximity, we shall refer to the edges of 3-graphs as \emph{faces} to avoid confusion. Finally, in what follows, pairs and triples refer respectively to unordered two-element and three-element sets; again, we abuse notation slightly and abbreviate a pair $\{x,y\}$ as $xy$, a triple $\{x, y, z\}$ as $xyz$, and so on. 

It will be convenient to work with 3-partite 3-graphs; the following fact facilitates this, and follows from an easy averaging argument.

\begin{proposition}\label{triedges} Any 3-graph on $3n$ vertices with $m$ edges contains a 3-partite subgraph with vertex classes of size $n$ and at least $2m/9$ edges. \qed
\end{proposition}

Now, let $H$ be a fixed 3-graph and let $G$ be a 3-partite 3-graph whose three vertex classes $X$, $Y$ and $Z$ are each of size $n$. Our strategy to construct a homeomorph of $H$ in $G$ will involve gluing various building blocks together appropriately; below, we introduce the notions we require to execute this strategy.

First, we shall construct an \emph{auxiliary 2-graph $\CS(H)$} from $H$ that will be helpful in finding a homeomorphic copy of $H$ in $G$. The construction of $\CS(H)$ from $H$, illustrated in Figure~\ref{fig:S(H)}, is as follows: first, for each pair $xy$ that is contained in some face of $H$, we introduce a new vertex $u = u_{xy}$ in $\CS(H)$ and add the edges $xu$ and $yu$ to $\CS(H)$; then, for each face $xyz$ of $H$, we introduce a new vertex $u = u_{xyz}$ in $\CS(H)$ and add the edges $xu$, $yu$ and $zu$ to $\CS(H)$. 

We record a few facts about $\CS(H)$ below.
\begin{enumerate}
\item Each face $xyz$ of $H$ gives rise to three specific 4-cycles in $\CS(H)$, namely the 4-cycles $\{x,u_{xy},y,u_{xyz}\}$, $\{y,u_{yz},z,u_{xyz}\}$ and $\{z,u_{zx},x,u_{xyz}\}$; we call the 4-cycles of this form the \emph{special} 4-cycles of $\CS(H)$.
\item $\CS(H)$ is bipartite, with the set $V_1$ of the original vertices of $H$ and the set $V_2$ of the new vertices added in the construction of $\CS(H)$ forming a bipartition. 
\item The degree of any vertex of $\CS(H)$ in $V_2$ is at most 3.
\end{enumerate}

\begin{figure}
	\centering
	\begin{tikzpicture}[scale=5, every node/.style={scale=1}]
	\draw[fill,gray,opacity=0.2] (0,0)--(1,0)--(1/2,0.866)--(0,0);
	\draw (0,0)--(1,0)--(1/2,0.866)--(0,0);
	\draw (2,0)--(3,0)--(2.5,0.866)--(2,0);
	\draw (0,0) node[] {$\bullet$};
	\draw (1,0) node[] {$\bullet$};
	\draw (1/2,0.866) node[] {$\bullet$};
	\draw (2,0) node[red] {$\bullet$};
	\draw (3,0) node[red] {$\bullet$};
	\draw (2.5,0.866) node[red] {$\bullet$};
	
	\draw (2.25,0.866/2) node[blue] {$\bullet$};
	\draw (2.75,0.866/2) node[blue] {$\bullet$};
	\draw (2.5,0) node[blue] {$\bullet$};
	
	\draw (2.5,0.289) node[blue] {$\bullet$};
	
	\draw (2,0)--(2.5,0.289);
	\draw (3,0)--(2.5,0.289);
	\draw (2.5,0.866)--(2.5,0.289);
	
	\draw[->,  line width=0.3mm] (1.2,0.433) -- (1.8,0.433);
	\draw (0.5,-0.15) node[] {$ H$};
	\draw (2.5,-0.15) node[] {$ \CS(H)$};

	\end{tikzpicture}
	\caption{The construction of $\CS(H)$ from $H$; the colours red and blue describe the bipartition of $\CS(H)$.}\label{fig:S(H)}
\end{figure}
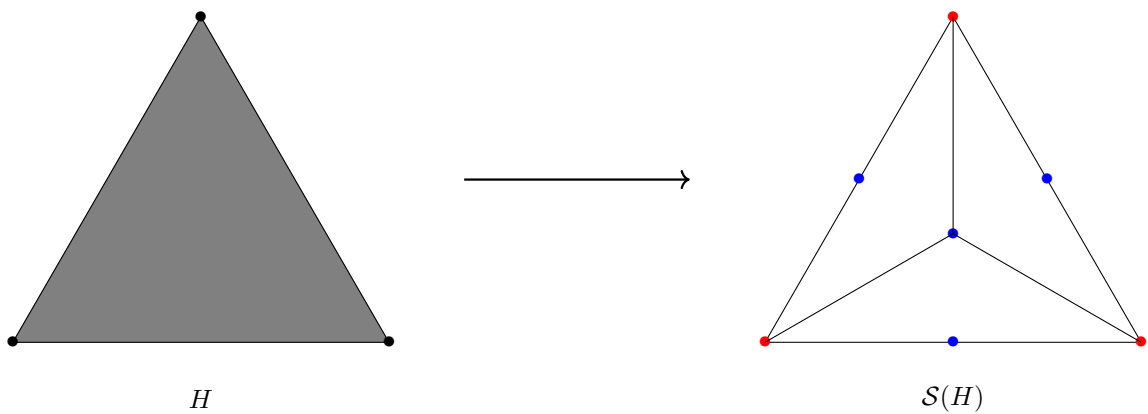

Next, we describe the structures within $G$ that will serve as building blocks in constructing a homeomorph of $H$. The \emph{link graph $\CL_z$} of a vertex $z\in Z$ is the bipartite graph between $X$ and $Y$ whose edges are those pairs $xy$ for which $xyz$ is a face of $G$. Notice that a 4-cycle in the link graph $\CL_z$ corresponds to four faces of $G$ (all sharing the vertex $z$) that, taken together, are homeomorphic to a disk; we call such a collection of four faces of $G$ a \emph{4-disk} with \emph{centre z}, and call the associated 4-cycle in the link graph $\CL_z$ the \emph{boundary} of the 4-disk.

Notice that a fixed 4-cycle in the complete bipartite graph between $X$ and $Y$ may be the boundary of anywhere between $0$ and $n$ different 4-disks in $G$. We set $K_H=3v(H)^3$, and call a 4-cycle between $X$ and $Y$ 
\begin{enumerate}
\item \emph{$H$-admissible} if the cycle is the boundary of more than $K_H$ different $4$-disks in $G$, and
\item \emph{$H$-forbidden} if this cycle is the boundary of between $0$ and $K_H$ different $4$-disks in $G$.
\end{enumerate}
The definitions of admissible and forbidden 4-cycles are motivated by the following observation. As noted earlier, each face of $H$ corresponds to three special 4-cycles in $\CS(H)$, and if we manage to find a copy of $\CS(H)$ in $X \times Y$ with the property that all its 4-cycles form boundaries of 4-disks in $G$ with distinct centres, then we may glue the corresponding 4-disks together to obtain a homeomorph of $H$ in $G$. 

We shall rely on $H$-admissible 4-cycles between $X$ and $Y$ to build a homeomorph of $H$ in $G$. First, assuming $G$ has sufficiently many edges, we shall show that we may pass to a subgraph $G'$ of $G$ in which most of the 4-cycles in $X \times Y$ are $H$-admissible. Next, we shall show, using $G'$, that we may find a copy of  $\CS(H)$ between $X$ and $Y$ with the property that each of the $4$-cycles in this copy is $H$-admissible. Finally, since an $H$-admissible 4-cycle is contained in at least $3v(H)^3 \ge 3e(H)$ different link graphs, we will be able to ensure that we never re-use central vertices when gluing the appropriate 4-disks in $G'$ together to construct a homeomorph of $H$. 

\section{Proof of the main result}\label{sec:proof}

As before, let $H$ be a fixed $3$-graph, take $K_H=3v(H)^3$, and let $G$ be a 3-partite 3-graph whose three vertex classes $X$, $Y$ and $Z$ are each of size $n$.

Our first goal is to find a vertex $z \in Z$ whose link graph $\CL_z$ is sufficiently dense so as to contain many copies of the auxiliary 2-graph $\CS(H)$ defined in Section~\ref{sec:prelim}, and which has a small number of $H$-forbidden 4-cycles. In order to achieve this, we use a straightforward application of dependent random choice; with the set-up as above, we have the following claim.

\begin{lemma}\label{lem:pickz}
If $e(G) \ge Cn^{3-\delta}$, then there exists a vertex $z\in Z$ such that
\begin{enumerate}
\item $e(\CL_z) \ge (C/2)n^{2-\delta}$, and 
\item the number of $H$-forbidden 4-cycles in $\CL_z$ is at most $(2K_H/C)n^{1+\delta}e(\CL_z)$.
\end{enumerate}
\end{lemma}

\begin{proof}
	Select a vertex $z\in Z$ uniformly at random. It is clear that $\E[e(\CL_z)] = e(G)/n$. Note that the probability that any given $H$-forbidden 4-cycle is contained in $\CL_z$ is at most $K_H/n$. Therefore, writing $B_z$ for the number of $H$-forbidden 4-cycles in $\CL_z$, we have $\E[B_z] \le K_Hn^{3}$.
	
Putting the two estimates above together, we have
\[\E\left[ e(\CL_z) -(C/2) n^{2-\delta} -(C/2K_H)n^{-1-\delta}B_z)\right]\ge 0,\]
so there must exist a vertex $z \in Z$ for which we both have 
\[ e(\CL_z) -(C/2)n^{2-\delta}\ge 0\]
and
\[ e(\CL_z) -(C/2K_H)n^{-1-\delta}B_z\ge 0,\]
proving the claim.
\end{proof}

Our proof of Theorem~\ref{thm:main} will hinge around finding a copy of the auxiliary 2-graph $\CS(H)$ within the link graph $\CL_z$ promised by Lemma~\ref{lem:pickz} while avoiding $H$-forbidden 4-cycles. To find this copy, we first show that we can pass to a large subset $Y'\subset Y$ within which almost all pairs and triples are well-behaved. To quantify what it means to be well-behaved, we make the following definitions.
\begin{enumerate}
\item We call a pair $y_1y_2$ of vertices in $Y$ \emph{good} if 
\[|\Gamma(y_1,y_2)|\ge n^{1-2\eps}\] 
and there are at most 
\[(K_H/C)n^{1-3\eps}|\Gamma(y_1,y_2)|\] 
$H$-forbidden 4-cycles containing both $y_1$ and $y_2$, and \emph{bad} otherwise.
\item We call a triple $y_1y_2y_3$ of vertices in $Y$ \emph{good} if 
\[|\Gamma(y_1, y_2, y_3)| \ge n^{1-3\eps},\]
and \emph{bad} otherwise.
\end{enumerate} 
With this set-up, we next show the following.

\begin{lemma}\label{lem:pickx}
Let $\eps\le 1/5$, $C \ge 1$ and let $\CL_z$ be a bipartite graph between $X$ and $Y$ with $(C/2)n^{2-\eps}$ edges in which the number of $H$-forbidden 4-cycles is at most $K_Hn^{3+1/5-\eps}$. Then there is a subset $Y'$ of $Y$ of size at least $n^{1-\eps}/4$ within which
\begin{enumerate} 
\item at most $(400/C)\binom{|Y'|}{2}$ pairs are bad, and
\item at most $(600/C)\binom{|Y'|}{3}$ triples are bad.
\end{enumerate}
\end{lemma}
\begin{proof}
To prove the lemma, we start by selecting a vertex $x\in X$ uniformly at random. Clearly, we have 
\[\E[|\Gamma(x)|] = (C/2)n^{1-\eps}.\] 

Notice that if a pair $y_1y_2$ in $Y$ is bad, then either
\begin{enumerate}[(a)]
\item\label{b1} $|\Gamma(y_1,y_2)| < n^{1-2\eps}$, or
\item\label{b2} the number of $H$-forbidden 4-cycles through $y_1$ and $y_2$ is at least $(K_H/C)n^{1-3\eps}|\Gamma(y_1,y_2)|$,
\end{enumerate}
or possibly both.

First, given a bad pair $y_1y_2$ in $Y$ for which~\ref{b1} holds, since $|\Gamma(y_1,y_2)| < n^{1-2\eps}$, the probability that both $y_1$ and $y_2$ belong to $\Gamma(x)$ is at most $n^{-2\eps}$; hence, the number $P_1$ of such pairs surviving in $\Gamma(x)$ satisfies 
\[\E[P_1] \le n^{2-2\eps}.\]

Next, let $\mathcal{Q}$ denote the set of bad pairs $y_1y_2$ for which~\ref{b2} holds, so each pair $y_1y_2 \in \mathcal{Q}$ lies in at least $(K_H/C)n^{1-3\eps}|\Gamma(y_1,y_2)|$ $H$-forbidden 4-cycles. Since the total number of $H$-forbidden 4-cycles in $\CL_z$ is at most $K_Hn^{3+1/5-\eps}$, we get
	\[\sum_{y_1y_2\in \mathcal{Q}} (K_H/C)n^{1-3\eps}|\Gamma(y_1,y_2)|\le K_Hn^{3+1/5-\eps},\]
	which implies that
	\[\sum_{y_1y_2\in \mathcal{Q}}|\Gamma(y_1,y_2)|\le Cn^{2+1/5+2\eps}.\]
It follows that the number $P_2$ of pairs in $\mathcal{Q}$ surviving in $\Gamma(x)$ satisfies
	\[ \E[P_2] = \sum_{y_1y_2\in \mathcal{Q}} \frac{|\Gamma(y_1,y_2)|}{ n} \le Cn^{1+1/5+2\eps}.\]

Thus, the total number $P_x$ of bad pairs surviving in $\Gamma(x)$, which is clearly at most the sum $P_1+P_2$, satisfies 
	\[\E[P_x] \le \E[P_1] + \E[P_2] \le Cn^{1+1/5+2\eps}+n^{2-2\eps}\le (1+C)n^{2-2\eps};\]
	here, the last inequality relies on the fact that $\eps \le 1/5$.

	Finally, given a bad triple $y_1y_2y_3$ in $Y$, since $|\Gamma(y_1,y_2,y_3)| < n^{1-3\eps}$, the probability that this triple survives in $\Gamma(x)$ is at most $n^{-3\eps}$. Writing $T_x$ for the number of bad triples surviving in $\Gamma(x)$, we again have
	\[\E[T_x] \le n^{3-3\eps}. \]

Putting the above estimates together, we get
	\[ \mathbb{E}\left[|\Gamma(x)|- \frac{Cn^{1-\eps}}{4}-\frac{CP_x}{12(1+C)n^{1-\eps}}-\frac{CT_x}{6n^{2-2\eps}}\right]\ge0,\]
	which in particular implies that there is some $x \in X$ for which we have
	\begin{enumerate}[(A)]
	\item\label{i1} $|\Gamma(x)|\ge Cn^{1-\eps} / 4$,
	\item\label{i2} $|\Gamma(x)|\ge CP_x/12(1+C)n^{1-\eps}$, and
	\item\label{i3} $|\Gamma(x)|\ge CT_x/6n^{2-2\eps}$.
	\end{enumerate}
Multiplying the inequality in~\ref{i2} by the one in~\ref{i1}, the inequality in~\ref{i3} by the square of the one in~\ref{i1}, we see that for this choice of $x \in X$, we have
	\[ P_x\le \left(48(1+C)/C^2\right)|\Gamma(x)|^2 \le (400/C)\binom{|\Gamma(x)|}{2},\]
	and
	\[ T_x\le \left(96/C^{3}\right)|\Gamma(x)|^3 < \left(600/C^{3}\right)\binom{|\Gamma(x)|}{3}, \]
provided $n$ is sufficiently large. Taking $Y'=\Gamma(x)$ for this choice of $x$ proves the claim.
\end{proof}

We are now ready to put these two lemmas together to prove our main result.

\begin{proof}[Proof of Theorem~\ref{thm:main}]
Our goal given is to find a homeomorph of a given 3-graph $H$ in any large 3-graph $G$ with sufficiently many faces. Appealing to Proposition~\ref{triedges}, we start by assuming that $G$ is a 3-partite 3-graph whose three vertex classes $X$, $Y$ and $Z$ are each of size $n$, and which has at least $Cn^{3-1/5}$ faces for some suitably large constant $C$ depending on $H$ alone. As described in Section~\ref{sec:prelim}, we shall work with the auxiliary 2-graph $\CS(H)$ to find a homeomorph of $H$ in $G$. 

First, we apply Lemma~\ref{lem:pickz} to $G$ with $\delta=1/5$. This gives us a vertex $z \in Z$ whose link graph $\CL_z$ contains $(C/2)n^{2-\eps}$ edges for some $\eps \le \delta = 1/5$ in which the number of $H$-forbidden 4-cycles is at most 
\[(2K_H/C)n^{1+1/5}e(\CL_z) = K_Hn^{3+1/5 - \eps}.\] 
This link graph satisfies the requirements of Lemma~\ref{lem:pickx}, so we apply the lemma to pass to a subset $Y' \subset Y$ within which most pairs and most triples are good.

Recall that $\CS(H)$ is bipartite and admits a bipartition $(V_1,V_2)$ where each vertex in $V_2$ has degree at most 3, and where $V_1$ is in fact the original set of vertices of $H$. 

We shall first embed the vertices of $V_1$ into $Y'$ in such a way that no embedded pair is bad and no embedded triple is bad. In order to show that this is possible, we note that the proportion of bad pairs in $Y'$ is at most $400/C$ and the proportion of bad triples in $Y'$ is at most $600/C^3$. We define a 3-graph $\CD(Y')$ on the vertex set $Y'$ whose edges are those that are potentially problematic for our embedding, i.e., those triples $y_1y_2y_3$ which are either bad, or for which one of the pairs $y_1y_2$, $y_2y_3$ or $y_1y_3$ is bad. The density of this 3-graph $\CD(Y')$ is at most $1200/C+600/C^{3}\le 2000/C$. 

Our goal now is to find a complete 3-graph on $|V_1|$ vertices in the complement of $\CD(Y')$, since the existence of such a subgraph enables us to inject $V_1$ into $Y'$ whilst avoiding all bad pairs and bad triples. A bound of de Caen~\citep{deCaen} shows that a copy of the complete 3-graph $K_t^{3}$ on $t$ vertices can be found in any $3$-graph on $n$ vertices of density at least \[1-\binom{t-1}{2}^{-1},\] 
provided $n$ is sufficiently large. Therefore, we may find our embedding (again, assuming that $n$ is sufficiently large) provided that 
\[2000/C\le \binom{|V_1|}{2}^{-1},\] 
which we may ensure by taking $C \ge 1000v(H)^2$.

It remains to find an embedding of the vertices of $V_2$ into $X$. For each vertex $u$ of degree 3 in $V_2$ that we need to embed into $X$, we have a choice of $n^{1-3\eps}$ vertices in the common neighbourhood of its three already-embedded neighbours from $V_1$; we choose its image from these candidates uniformly at random. Similarly, for each vertex $v$ of degree 2 in $V_2$, we choose its image uniformly at random from the $n^{1-2\eps}$ vertices in the common neighbourhood of its two already-embedded neighbours from $V_1$. 

The probability that this embedding is not proper, i.e., that some two vertices in $V_2$ get mapped to the same vertex in $X$, is at most $|V_2|^2 n^{3\eps -1} < 1/2$, provided $n$ is large (since $\eps \le 1/5$ and $|V_2| \le 10 v(H)^3$). 

We shall next show that for this embedding, the probability of some special $4$-cycle in $\CS(H)$ mapping to an $H$-forbidden $4$-cycle is also at most $1/2$. Since the number of special $4$-cycles in $\CS(H)$ is $3e(H)$, it suffices to show for each special 4-cycle $\CC$ in $\CS(H)$ that the probability of its image being $H$-forbidden is at most $1/(6e(H))$. Let $y_1$ and $y_2$ be the images of vertices of $\CC$ in $V_1$ (which have been fixed earlier deterministically), and consider $u'$ and $v'$, the (random) images of the two vertices $u$ and $v$ of $\CC$ from $V_2$ whose degrees in $\CS(H)$ are respectively 3 and 2. Let $y_3$ be the vertex in $Y'$ so that $u'$ is chosen uniformly at random from $\Gamma(y_1,y_2,y_3)$. Suppose for a contradiction that the probability of the image of $\CC$ being $H$-forbidden is at least $1/(6e(H))$. Then this implies that at least a $1/(6e(H))$ proportion of the 4-cycles formed by taking $y_1$ and $y_2$, together with a vertex $x_1\in \Gamma(y_1,y_2,y_3)$ and a vertex $x_2 \in \Gamma(y_1,y_2)$ are $H$-forbidden. This leads us to conclude that the number of $H$-forbidden 4-cycles through $y_1$ and $y_2$ is at least 
\[(1/6e(H))n^{1-3\eps}|\Gamma(y_1,y_2)|.\] However, if $1/(6e(H)) \ge K_H/C$, then this would imply that $y_1y_2$ is a bad pair, a contradiction; this final inequality may be ensured by taking $C \ge 18v(H)^6$, since $e(H)\le v(H)^3$ and $K_H=3v(H)^3$.
	
	We have shown that it is possible to embedded $\CS(H)$ into $\CL_z$ in such a way that all of the special 4-cycles in this embedding are $H$-admissible. This embedding extends to a homeomorph of $H$ inside $G$ as follows. For each 4-cycle $\CC$ in $\CL_z$ that is the image of some special 4-cycle of $\CS(H)$, we claim that we may choose a unique vertex $z(\CC) \in Z$ such that $\CC$ is also contained in the link graph $\CL_{z(\CC)}$: indeed, $\CC$ is $H$-admissible, so there are at least $K_H = 3v(H)^3\ge 3e(H)$ choices for $z(\CC)$. We then use $z(\CC)$ to turn each of the embedded special 4-cycles $\CC$ in $\CL_z$ into a 4-disk in $G$, noting that these 4-disks all have distinct centres; the result is a homeomorph of $H$ in $G$.
\end{proof}

\section{Conclusion}\label{sec:conc}
Below, we address some of the limitations of our approach to finding homeomorphs of a fixed target 3-graph $H$, as well as some potential avenues for improvement.

It seems plausible that the exponent of $3-1/5$ that we obtain may be improved somewhat by a more judicious application of the methods developed here. However, the ideas developed in this paper reach a bottleneck, conjecturally, at the exponent of $3-1/4$. This is because it is believed~\citep{Mubayi} that there exist $n$-vertex 3-graphs with $\Omega(n^{3-1/4})$ edges that do not contain any octahedra, though the best constructions presently known, see~\citep{katz}, only manage $\Omega(n^{3-1/3})$ edges. If a 3-graph does not contain any octahedra, then our approach based on $H$-admissible 4-cycles falls apart, since if all the 4-cycles in the link graphs are $H$-forbidden, then our method for extending $\CS(H)$ to a homeomorph of $H$ fails due to degeneracy concerns.

Another important fact to bear in mind is that while $\Omega(n^{5/2})$ edges guarantee a homeomorph of $S^2$ in any $n$-vertex 3-graph, the number of edges needed to guarantee a homeomorph of $S^2$ \emph{of bounded size} comes with an exponent strictly greater than $5/2$, as can be verified by a standard deletion argument applied to a (binomial) random 3-graph of the appropriate density. Our methods here end up finding bounded-size homeomorphs: indeed, we find a homeomorphic copy of $H$ that has $12e(H)$ edges. Any strategy that does not plan for the possibility of finding large homeomorphs of  $H$, i.e., of size unbounded in terms of $H$, cannot prove Conjecture~\ref{mainconj}.



We leave the reader with a reminder of the specialisation of Conjecture~\ref{mainconj} to the torus as reiterated by Linial~\citep{linialques1, linialques2}.
\begin{conjecture}
There is a $C >0$ such that any $3$-graph $G$ on $n$ vertices with at least $C n^{5/2}$ edges contains a homeomorph of the torus.
\end{conjecture}
An easy adaptation of the arguments of Brown, Erd\H{o}s, and S\'os~\citep{bes} to triple-pyramidal complexes (from double-pyramidal complexes) shows that an exponent of $3-1/3$ suffices for the torus, but improving on this bound remains an attractive starting point to Conjecture~\ref{mainconj} in its full generality.

\section*{Acknowledgements}
The first and second authors were partially supported by ERC Consolidator Grant 647678, and the third author wishes to acknowledge support from NSF grant DMS-1800521.
\bibliographystyle{amsplain}
\bibliography{homeo_turan}
\end{document}